\documentclass[12pt]{amsart}
\title{Large implies Henselian}
\author{Erik Walsberg}
\email{erik.walsberg@gmail.com}
\usepackage{amsmath, amssymb, amsthm, enumitem, comment}    % need for subequations
\usepackage[mathscr]{euscript}
\usepackage[Symbol]{upgreek}
\usepackage{fullpage} 	% without this, will have wide math-article-paper margins
\usepackage{hyperref}
\usepackage{lineno}
\usepackage[all]{xy}
\usepackage{centernot}
\usepackage{tikz-cd}

\usepackage{xcolor}

\DeclareFontFamily{U}{fsy}{}
\DeclareFontShape{U}{fsy}{m}{n}{<->s*[.9]psyr}{}
\DeclareSymbolFont{der@m}{U}{fsy}{m}{n}
\DeclareMathSymbol{\der}{\mathord}{der@m}{182}

\DeclareFontFamily{U}{BOONDOX-calo}{\skewchar\font=45 }
\DeclareFontShape{U}{BOONDOX-calo}{m}{n}{
  <-> s*[1.05] BOONDOX-r-calo}{}
\DeclareFontShape{U}{BOONDOX-calo}{b}{n}{
  <-> s*[1.05] BOONDOX-b-calo}{}
\DeclareMathAlphabet{\mathcalboondox}{U}{BOONDOX-calo}{m}{n}
\SetMathAlphabet{\mathcalboondox}{bold}{U}{BOONDOX-calo}{b}{n}
\DeclareMathAlphabet{\mathbcalboondox}{U}{BOONDOX-calo}{b}{n}

\DeclareSymbolFont{imag@m}{OT1}{cmr}{m}{ui}
\DeclareMathSymbol{\imag}{\mathord}{imag@m}{105}

\DeclareMathOperator*{\forkindep}{\raise0.2ex\hbox{\ooalign{\hidewidth$\vert$\hidewidth\cr\raise-0.9ex\hbox{$\smile$}}}}

\newcommand{\Sa}[1]{\ensuremath{\mathscr{#1}}}

\newcommand{\mfrak}{\mathfrak{m}}

\newcommand{\Gal}{\operatorname{Gal}}
\newcommand{\Spec}{\operatorname{Spec}}
\newcommand{\Hom}{\operatorname{Hom}}

\newcommand{\Frac}{\operatorname{Frac}}

\newcommand{\Chara}{\operatorname{Char}}
\newcommand{\jac}{\operatorname{Jac}}

\newcommand{\prc}{\mathrm{PRC}}
\newcommand{\pac}{\mathrm{PAC}}

\newtheorem*{claim-star}{Claim}
\newtheorem{theorem}{Theorem}[section] % numbered like the section
\newtheorem{lemma}[theorem]{Lemma}

\newtheorem{prop-def}[theorem]{Proposition-Definition}
\newtheorem{corollary}[theorem]{Corollary}
\newtheorem{fact}[theorem]{Fact}
\newtheorem{fact-eh}[theorem]{Fact(?)}

\newtheorem{proposition}[theorem]{Proposition}
\newtheorem{proposition-eh}[theorem]{Proposition(?)}
\newtheorem*{theorem-star}{Theorem}
\newtheorem*{conjecture-star}{Conjecture}
\newtheorem*{lemma-star}{Lemma}

\newtheorem*{lemma6.2alt}{Lemma \ref*{confusing-v2}$'$}

\theoremstyle{definition}

\theoremstyle{remark}

\newcommand{\A}{\mathbb{A}}

\newcommand{\Q}{\mathbb{Q}}
\newcommand{\R}{\mathbb{R}}
\newcommand{\K}{\mathbb{K}}

\newcommand{\C}{\mathbb{C}}

\newcommand{\cT}{\mathscr{T}}

\newcommand{\meno}{\medskip \noindent}

\begin{document}

\title{Nash maps over large fields}
\maketitle

\begin{abstract}
In this short note we give some corollaries to the polynomial inverse function theorem for large fields.
We prove inverse and implicit function theorems for Nash maps over large fields, characterize large fields as fields satisfying inverse or implicit function theorems, give inverse and implicit function theorems for gt-henselian field topologies, and show that definable functions in various logically tame fields of characteristic zero are generically Nash.
We prove a version of Krasner's lemma for large fields and describe how Nash functions give a natural proof of the well-known fact that a large field $K$ is existentially closed in $K(\!(t)\!)$.
\end{abstract}

\section{Introduction}
Let $K$ be a field.
Then $K$ is \textbf{large} if it satisfies one of the following equivalent conditions.
\begin{enumerate}
\item Any smooth $1$-dimensional $K$-variety with a $K$-point has infinitely many $K$-points.
\item If $f\in K[x,y]$ and $(a,b)\in K^2$ satisfy $f(a,b)=0\ne\der f/\der y (a,b)$, then $f$ has infinitely many zeros in~$K^2$.
\item The \'etale-open topology on $K$ is not discrete.
\item $K$ is existentially closed in $K(\!(t)\!)$.
\item $K$ is elementarily equivalent to the fraction field of a henselian local domain that is not a field. 
(See \cite[Thm.~A]{large->henselian}.)
\end{enumerate}
Let $V$ range over $K$-varieties.
The \textbf{\'etale-open topology} on $V(K)$ is the topology with basis the collection of sets of the form $f(W(K))$ for a $K$-variety $W$ and an \'etale morphism $f\colon W\to V$.
This topology was introduced in~\cite{firstpaper}.
Examples of large fields include separably closed fields, real closed fields, henselian fields, fields with pro-$p$ absolute Galois group,  $\pac$ fields (e.g. pseudofinite fields and infinite algebraic extensions of finite fields), $\prc$ fields (e.g. the maximal totally real extension of $\Q$), and fields that satisfy a local-global principle with respect to a set of places~\cite[1.B]{Pop-little}.
The following polynomial inverse function theorem for large fields was proven in~\cite[Thm.~B]{large->henselian}.

\begin{fact}\label{fact:etale-homeo}
Suppose that $K$ is large and not separably closed.
If $V \to W$ is an \'etale morphism of $K$-varieties then $V(K) \to W(K)$ is a local homeomorphism in the \'etale-open topology.
\end{fact}

The assumption that $K$ is not separably closed is necessary as the \'etale-open topology over a separably closed field is the Zariski topology.

\medskip
One can show that $\R$ is large by assuming that $f \in \R[x,y]$ and $(a,b) \in \R^2$ satisfy the conditions of (2) above and applying the implicit function theorem to produce a neighborhood $U$ of $a$ and analytic $s \colon U \to \R$ such that $s(a) = b$ and $f(a^*, s(a^*)) = 0$ for all $a^* \in U$.
We show that in fact any non-separably closed large field satisfies an implicit function theorem with respect to the \'etale-open topology and we characterize large fields as those fields which satisfy an implicit function theorem, see Proposition~\ref{prop:nash-implicit} and Section~\ref{section:char}.
Both parts involve Nash functions.

\medskip
Suppose $O$ is a semialgebraic open subset of $\R^m$ and fix $f\colon O \to \R^m$.
Then $f$ is Nash if $f$ is both analytic and algebraic, equivalently if $f$ is $C^\infty$ and semialgebraic~\cite[\S{}8]{real-algebraic-geometry}.
Artin-Mazur showed that $f$ is Nash if and only if there is an $\R$-variety $V$, an \'etale morphism $e\colon V \to \A^m$, and a morphism $h\colon V \to \A^n$ such that $e$ gives a homeomorphism $P \to O$ for some open $P \subseteq V(\R)$, and $f$ agrees with $h\circ (e|_P)^{-1}$~\cite[Thm.~8.4.4]{real-algebraic-geometry}.

\medskip
Suppose that $K$ is large and not separably closed and let $V, V^*$ be $K$-varieties and $O, O^*$ be a nonempty \'etale-open subset of $V(K), V^*(K)$, respectively.
Then $f \colon O \to O^*$ is a \textbf{Nash map} if there is a $K$-variety $X$, \'etale morphism $e \colon X \to V$, morphism $h \colon X \to W$, and \'etale-open subset $P$ of $X(K)$ such that $e$ restricts to a homeomorphism $P \to O$ and $f = h \circ (e|_P)^{-1}$.
We refer to $(X,e,h,P)$ as \textbf{Nash data} for $f$.
A {\bf Nash function} is a Nash map with codomain $K$.
For example $\sqrt[n]{x}$ is a Nash function on an open neighborhood of $1$ if and only if $n$ is not divisible by the characteristic of $K$.

\medskip
We show how Nash functions can be used to naturally prove the well-known fact that a large field $K$ is existentially closed in $K(\!(t)\!)$ in Section~\ref{section:(())}.

\subsection{Conventions and background}
Throughout $K$ is a field.
A $K$-variety is a separated reduced scheme of finite type over $K$.
We let $V,W,V^*,W^*$, etc, be $K$-varieties and let $V(K)$ be the set of $K$-points of a $K$-variety.
A morphism $V \to W$ is a morphism of $K$-varieties.
We let $T_p V$ be the Zariski tangent space of a smooth $K$-variety $V$ at a $K$-point $p$.
Recall that if $W$ is smooth then $f\colon V \to W$ is \'etale if and only if $V$ is smooth and $T_p V \to T_{f(p)}W$ is an isomorphism for all $p \in V$.
We let $\A^n$ be $n$-dimensional affine space over $K$, i.e. $\A^n = \Spec K[x_1,\ldots,x_n]$.
Recall that $\A^n(K)$ is canonically identified with $K^n$.

\medskip
We now recall some background on the \'etale-open topology.
Equip $V(K)$ and $W(K)$ with the \'etale-open topology.
Then we have:
\begin{enumerate}
\item $V(K) \to W(K)$ is continuous for any $V \to W$.
\item $V(K) \to W(K)$ is a topological closed embedding when $V \to W$ is a scheme-theoretic closed immersion.
\item $V(K) \to W(K)$ is a topological open embedding when $V \to W$ is a scheme-theoretic open immersion.
\end{enumerate}

Fact~\ref{fact:slice}(1) is immediate from (1) and Fact~\ref{fact:slice}(2) follows from Fact~\ref{fact:slice}(1).

\begin{fact}\label{fact:slice}
Suppose that $U$ is an \'etale-open subset of $V(K) \times W(K) = (V \times W)(K)$.
Then we have the following.
\begin{enumerate}[leftmargin=*]
\item If $a \in V(K)$ then $\{ b \in W(K) : (a,b) \in O\}$ is an \'etale-open subset of $W(K)$.
\item The image of $U$ under the projection $V(K) \times W(K) \to W(K)$ is \'etale-open.
\end{enumerate}
\end{fact}

Fact~\ref{fact:z-dense} is \cite[Prop.~5]{topological_proofs}.

\begin{fact}\label{fact:z-dense}
If $K$ is large and $V$ is smooth and irreducible then any nonempty \'etale-open subset of $V(K)$ is Zariski dense in $V$.
\end{fact}

Fact~\ref{fact:hd} is \cite[Thm.~7.6]{firstpaper}.

\begin{fact}\label{fact:hd}
If $K$ is not separably closed then the \'etale-open topology on $V(K)$ is hausdorff when $V$ is quasi-projective.
\end{fact}

We recall some basic facts about Nash functions.
Fact~\ref{fact:nash basic} is \cite[Prop.~7.1, 7.2, 7.3]{large->henselian}.

\begin{fact}\label{fact:nash basic}
\hspace{.00001cm}
\begin{enumerate}[leftmargin=*]
\item Nash maps are closed under composition.
\item If $V,V^*_1,\ldots,V^*_n$ are smooth $K$-varieties and $O\subseteq V(K), O^*_1\subseteq V^*_1(K), \ldots, O^*_n\subseteq V^*_n(K)$ are nonempty \'etale-open sets then a map $O \to O^*_1\times\cdots\times O^*_n$ is Nash if and only if each component function $O \to O^*_i$ is Nash.
\item If $O$ is a nonempty \'etale-open subset of $V(K)$ then Nash functions $O \to K$ form a $K$-algebra under the pointwise operations.
\end{enumerate}
\end{fact}

Actually, Fact~\ref{fact:nash basic} is only proven for Nash maps on smooth varieties, but the argument goes through in the more general case.
Here (1) shows that Nash maps form a category and (2) shows that the product in this category is the set-theoretic product.
Let $\Sa N_p$ be the $K$-algebra of germs of Nash functions at $p$.

\begin{fact}\label{fact:inclusion}
Suppose that $K$ is large, $V$ is a smooth and irreducible, $p \in V(K)$, and $\Sa O_p$ is the local ring of $V$ at $p$.
Then $f \mapsto f|_{V(K)}$ gives a $K$-algebra embedding $\Sa O_p \to \Sa N_p$.
\end{fact}

Fact~\ref{fact:inclusion}, a consequence of Fact~\ref{fact:z-dense}, is proven in \cite[Lemma~7.6]{large->henselian}.

\begin{fact}\label{fact:alg power series}
Let $V$ be a smooth $K$-variety and $p \in V(K)$.
Then $\Sa N_p$ is a local $K$-algebra with maximal ideal $\mfrak_p = \{f \in \Sa N_p : f(p) = 0\}$ and residue field $K$.
Furthermore $\Sa N_p$ is isomorphic as a $K$-algebra to the ring of algebraic power series over $K$ in $d$ variables for $d$ the dimension of the irreducible component of $V$ containing $p$.
\end{fact}

Finally, Fact~\ref{fact:alg power series} is \cite[Thm.~7.9]{large->henselian}.

\section{Existential closedness in $K(\!(t)\!)$}\label{section:(())}
The field $K$ is {\bf existentially closed} in a $K$-algebra $A$ if  any finite system of polynomial equations and inequations with coefficients from $K$ that has a solution in $A$ already has a solution in $K$.
It is an important fact that a large field $K$ is existentially closed in $K(\!(t)\!)$.
In particular, this is crucially used in the solution of the regular inverse Galois problem over large fields~\cite{pop-embedding}.
We explain how to show that large $K$ is existentially closed in $K(\!(t)\!)$ via Nash functions.
See \cite[Prop.~2.4]{Pop-little} for the standard proof.

\medskip
Fact~\ref{fact:ec} follows easily from the definition of existential closedness.

\begin{fact}\label{fact:ec}
If $D$ is a domain extending $K$ then $K$ is existentially closed in $D$ if and only if $K$ is existentially closed in $\Frac(D)$.
\end{fact}

We now prove an easy lemma on Nash functions.

\begin{lemma}\label{lem:nashlem}
Suppose that $K$ is large and not separably closed, $V$ is a smooth $K$-variety, $p \in V(K)$, and $f,g$ are distinct germs of Nash functions at $p$.
Then there is an \'etale-open neighborhood $O$ of $p$ such that $f,g$ are both defined on $O$ and $\{ a \in O : f(a)\ne g(a)\}$ contains a nonempty Zariski open subset of $O$.
\end{lemma}

In particular if $\dim V = 1$ then we have $f(a) \ne g(a)$ for all  $a \in O \setminus \{p\}$ for $O$ a sufficiently small \'etale-open neighborhood of $p$.

\begin{proof}
After replacing $f$ with $f - g$ we suppose  $g$ is zero and $f$ is not.
Let $O$ be an \'etale-open neighborhood of $p$ and $f^*\colon O \to K$ be a Nash function whose germ at $p$ is $f$.
Let $(X,e,h,P)$ be Nash data for $f^*$.
Then $X$ is smooth as $V$ is smooth and $X \to V$ is \'etale.
After possibly shrinking $O$ we may replace $X$ with an irreducible component of $X$ and therefore suppose that $X$ is irreducible.
Let $Y$ be the subvariety of $X$ given by $h = 0$.
Then $h$ is not constant zero hence $\dim Y < \dim X$ as $X$ is irreducible.
Let $Z$ be the Zariski closure of $e(Y)$ in $V$.
Then $\dim Z = \dim e(Y) \le \dim Y < \dim X$.
Furthermore $\dim X= \dim V$ as $e$ is \'etale.
Now $\{ a \in O : f(a)\ne 0\}$ contains $O \cap (V\setminus Z)(K)$, and $O\cap (V\setminus Z)(K)\ne\varnothing$ by Fact~\ref{fact:z-dense}.
\end{proof}

Now suppose that $K$ is large and not separably closed.
We show that $K$ is existentially closed in $K(\!(t)\!)$.
By Fact~\ref{fact:ec} it is enough to show that $K$ is existentially closed in $K[[t]]$.
Let $K[[t]]_\mathrm{alg}$ be the ring of algebraic power series over $K$ in one variable.
By the Artin approximation theorem $K[[t]]_\mathrm{alg}$ is existentially closed in $K[[t]]$~\cite{Artin_approx}.
Hence it is enough to show that $K$ is existentially closed in $K[[t]]_\mathrm{alg}$.
Fix $f_1,\ldots,f_n,g_1,\ldots,g_m \in K[x_1,\ldots,x_d]$ and consider the system $\mathcal{P}$ of equations and inequations given by
$$ f_i(x_1,\ldots,x_d) = 0 \ne g_j(x_1,\ldots,x_d) \quad \text{for all } 1 = 1,\ldots,n \text{ and } j = 1,\ldots,m.$$
Let  $(h_1,\ldots,h_d) \in K[[t]]_\mathrm{alg}^d$ be a solution of $\mathcal{P}$.
We identify $K[[t]]_\mathrm{alg}$ with the ring of germs of Nash functions on $K$ at the origin and therefore consider each $h_i$ to be a Nash function germ.
Each $g_j(h_1,\ldots,h_d)$ is also a Nash function germ.
By Lemma~\ref{lem:nashlem} there is an \'etale-open neighborhood $U\subseteq K$ of the origin so that each $h_i$ is defined on $U$ and each $g_j(h_1,\ldots,h_d)$ does not vanish on $U \setminus \{0\}$.
Hence $(h_1(a),\ldots,h_d(a))$ is a solution of $\mathcal{P}$ in $K^d$ for any $a \in U \setminus \{0\}$.

\medskip
The argument above only covers large fields that are not separably closed.
The separably closed case holds simply as separably closed fields are $\pac$, a $\pac$ field is existentially closed in any regular field extension, and $K(\!(t)\!)$ is a regular field extension of $K$.
Alternatively, if $K$ is separably closed and not the algebraic closure of a finite field, then we can take a non-trivial valuation $v$ on $K$, note that $v$ is henselian, develop a theory of Nash maps over $K$ with respect to $v$, and adapt the argument above.

% \medskip
% A {\it punctured Nash germ} is a Nash map with domain $U \setminus \{0\}$ for $U \subseteq K$ an \'etale-open neighborhood of the origin, where two punctured germs are considered equal if they agree on a sufficiently small neighborhood of the origin.
% The argument above yields the following.

% \begin{proposition}\label{prop:punctured}
% Suppose that $K$ is large and not separably closed and let $V$ be a $K$-variety.
% There is a canonical one-to-one correspondence between $V(K[[t]]_\mathrm{alg})$ and punctured Nash germs taking values in $V(K)$.
% \end{proposition}

\section{Nash maps}
{\bf We suppose throughout this section that $K$ is large and not separably closed.}

\subsection{The Nash Jacobian}
\label{section:nash-derivative}
If $f$ is a germ of Nash function on $K$ at the origin, then $f$ is identified with an algebraic power series, and the linear part of the power series should be the Jacobian of $f$ at the origin.
We give a commutative-algebraic definition of the Jacobian in general.
We define it as a functor so that the chain rule automatically holds.

\medskip
Given smooth $V,V^*$ and \'etale-open $O\subseteq V(K), O^*\subseteq V^*(K)$ we wish to define the Jacobian of a Nash map $O \to O^*$ at a point $p \in O$, this will be a linear map $T_p V \to T_{f(p)} V^*$.
Let $\mathtt{LNash}$ be the category of germs of Nash maps on smooth varieties.
The objects of $\mathtt{LNash}$ are pairs $(V(K),p)$ for a smooth $V$ and $p \in V(K)$, and an arrow $(V(K),p) \to (V^*(K),p^*)$ is a germ of a Nash map $f \colon O \to V^*(K)$ such that $f(p) = p^*$ where $O \subseteq V(K)$ is an open neighborhood of $p$.
% Let $\mathtt{LVar}_K$ be the category of germs of morphisms of smooth $K$-varieties at $K$-points.
% The objects of $\mathtt{LVar}_K$ are pairs $(V,p)$ with $p \in V(K)$ and $V$ smooth and an arrow $(V,p) \to (V^*,p^*)$ is a germ (with respect to the Zariski topology) of a $K$-variety morphism $f\colon V \to V^*$ such that $f(p) = p^*$.
% There is a natural functor $\mathtt{LVar}_K \to \mathtt{LNash}$ that takes $(V,p)$ to $(V(K),p)$ and takes $f$ to $f|_{V(K)}$.
% \medskip
% Lemma~\ref{lem:germ0} follows easily by Fact~\ref{fact:z-dense}.
% \begin{lemma}
% \label{lem:germ0}
% The functor $\mathtt{LVar}_K \to \mathtt{LNash}$ is faithful.
% \end{lemma}
% We therefore consider $\mathtt{LVar}_K$ to be a subcategory of $\mathtt{LNash}$.
Let $\mathtt{Vec}_K$ be the category of finite dimensional $K$-vector spaces.
% The usual algebraic Jacobian gives a functor $\mathtt{LVar}_K \to \mathtt{Vec}_K$.
% We extend this to a functor $\mathtt{LNash} \to \mathtt{Vec}_K$.

\meno
Let $\mathtt{Reg}$ be the category of  regular local $K$-algebras with residue field $K$ and local ring morphisms.
Fix $A \in \mathtt{Reg}$ with maximal ideal $\mfrak$ and let $\mfrak^2 = \{ ab : a,b \in \mfrak \}$.
The cotangent space of $A$ is $\mfrak/\mfrak^2$, considered as a $K$-vector space, and the tangent space of $A$ is the dual vector space $\Hom_{\mathtt{Vec}_K}(\mfrak/\mfrak^2,K)$.
The tangent space of $A$ is a $d$-dimensional $K$-vector space, where $d$ is the Krull dimension of $A$.
This gives a contravariant tangent space functor $\mathtt{Reg} \to \mathtt{Vec}_K$.
% The algebraic Jacobian is the composition of the canonical functor $\mathtt{LVar}_K \to \mathtt{Reg}$ with the tangent space functor $\mathtt{Reg} \to \mathtt{Vec}_K$.
Let $\Sa N_p$ be the ring of germs of Nash functions at $p \in V(K)$.
By Fact~\ref{fact:alg power series} $\Sa N_p$ is a $d$-dimensional regular local ring for $d$ the dimension of the irreducible component of $V$ containing $p$.
The inclusion $\Sa O_p \to \Sa N_p$ given by Fact~\ref{fact:inclusion} therefore gives a $K$-vector space isomorphism from $T_p V$ to $\Hom_{\mathtt{Vec}_K}(\mfrak_p/\mfrak^2_p, K)$.
Furthermore any Nash germ $f\colon (V(K),p) \to (V^*(K),p^*)$ gives a local $K$-algebra morphism $\Sa N_{p^*} \to \Sa N_p$ by $h \mapsto h \circ f$.
This gives a contravariant functor from $\mathtt{LNash}$ to $\mathtt{Reg}$.
The local Nash Jacobian is the functor $\mathtt{LNash} \to \mathtt{Vec}_K$ given by composing the contravariant functors $\mathtt{LNash} \to \mathtt{Reg}$ and $\mathtt{Reg} \to \mathtt{Vec}_K$.

\meno
Fix smooth $K$-varieties $V,V^*$ and \'etale-open subsets $O, O^*$ of $V(K), V^*(K)$, respectively.
Given a Nash map $f \colon O \to O^*$ and $p \in O$, we let $\jac_f(p)$ be the morphism $T_p V \to T_{f(p)} V^*$ associated to the germ of $f$ at $p$ via the functor $\mathtt{LNash} \to \mathtt{Vec}_K$ described above.
Fix Nash data $(X,e,h,P)$ for $f$.
Let $q$ be the unique element of $P$ such that $e(q) = p$.
The Nash germ $(X(K),q) \to (V(K),p)$ of $e$ is invertible.
By functorality $\jac_e(q)$ is an invertible linear map $T_q X \to T_p V$.
Functorality (i.e. the chain rule) also shows that
$$\jac_f(p) = \jac_h(q) \circ \jac_e(q)^{-1}.$$
Let $M_{n \times m}(K)$ be the set of $n \times m$ matrices with entries in $K$.
When $V = \A^m$ and $V^* = \A^n$ we identify $T_p V$, $T_p W$ with $K^m, K^n$, respectively, and therefore consider $\jac_f(p)$ to be an element of $M_{n \times m}(K)$.
We define $\der f_i / \der x_j (p)$ to be the $ij$-entry of $\jac_f(p)$.

\begin{proposition}
\label{prop:nash-diff}
Suppose that $f\colon O \to K^n$ is Nash for $O$ a nonempty open subset of $K^m$.
Then $\jac_f$ is a Nash map $O \to M_{n\times m}(K)$ and $\der f_i / \der x_j$ is a Nash function $O \to K$ for each $i = 1,\ldots,n$ and $j = 1,\ldots,m$.
\end{proposition}

We can iterate to define the higher partial derivatives $\der^{d_1 + \cdots + d_m} f / \der^{d_1} x_1 \ldots \der^{d_m} x_m$ and show that each partial derivative is a  Nash map $O \to K$.
We leave this to the reader.

\begin{proof}
The first and second claims are equivalent by Fact~\ref{fact:nash basic}(2).
We prove the first claim.
Let $(X,e,h,P)$ be Nash data for $f$.
For each $j \in \{1,\ldots,m\}$ we let $\der x_j$ be the usual algebraic vector field on $\A^m$.
Then each $\der x_j$ pulls back to an algebraic vector field on $X$ via $e$ which we denote by $\der \widehat{x}_j$.
As $e$ is \'etale $\der \widehat{x}_1 (p) , \ldots , \der \widehat{x}_m (p)$ is a basis for $T_p X$ at any $p \in X(K)$.
This gives coordinates on $T_p X$ for $p \in X(K)$.
We work with these coordinates on $T_p X$.
This allows us to consider all Jacobians to be matrices.
Let $q$ be the unique element of $P$ satisfying $e(q) = p$.
In these coordinates, $\jac_e (q)$ is the identity, hence $\jac_f(p) = \jac_h(q)$.
Therefore $\jac_h(q)$ is the $n \times m$ matrix with $ij$-entry $\der h_i / \der \widehat{x}_j (q)$.
Hence there is a morphism $j \colon X \to \A^{nm}$ such that $\jac_h$ is the restriction of $j$ to $X(K)$.
Therefore $(X,e,j,P)$ is also Nash data for $\jac_f$.
Hence $\jac_f$ is a Nash map.
\end{proof}

\subsection{Inverse and Implicit function theorems for Nash maps}\label{section:inv-imp}
Recall that if $V \to W$ is \'etale at $p \in V$ then $V \to W$ is \'etale on an open subvariety of $V$ containing $p$.
Hence Corollary~\ref{cor:etale-homeo} follows from Fact~\ref{fact:etale-homeo}.

\begin{corollary}\label{cor:etale-homeo}
If $f\colon V \to W$ is \'etale at $p \in V(K)$ then there are \'etale-open neighborhoods $p \in O \subseteq V(K)$ and $f(p) \in P \subseteq W(K)$ such that $f$ gives a homeomorphism $O \to P$.
\end{corollary}

A morphism $f\colon V \to W$ between smooth varieties is smooth at $p \in V$ if the induced map $T_p V \to T_{f(p)} W$ is surjective, and $f$ is {\bf smooth} if it is smooth at every $p \in V$.
It follows that $f$ is \'etale at $p$ if and only if $f$ is smooth at $p$ and the local dimension of $V$ at $p$ is equal to the local dimension of $W$ at $f(p)$, and $f$ is \'etale if and only if $f$ is smooth and $\dim V = \dim W$.
If $O, O^*$ is an \'etale-open subset of $V(K), V^*(K)$, respectively, with $V, V^*$ smooth then we say that a Nash map $f \colon O \to O^*$ is a {\bf submersion} at $p \in O$ if $\jac_p(f)$ is surjective and $f$ is a submersion when it is a submersion at every $p \in O$.

\begin{lemma}\label{lem:nash sub}
Let $V,V^*$ be smooth and, $O, O^*$ be a nonempty \'etale-open subset of $V(K), V^*(K)$, respectively.
Let $f\colon O \to O^*$ be a Nash map and fix $p \in O$.
Then the following are equivalent:
\begin{enumerate}[leftmargin=*]
\item $f$ is a submersion at $p$.
\item There is Nash data $(X, e, h, P)$ for $f$ such that $h$ is smooth at $q = (e|_P)^{-1}(p)$.
\item If $(X, e, h, P)$ is Nash data for $f$ then $h$ is smooth at $q = (e|_P)^{-1}(p)$.
\end{enumerate}
\end{lemma}

\begin{proof}
Clearly (3) implies (2).
Let $(X, e, h, P)$ be Nash data for $f$ and let $q = (e|_{P})^{-1}(p)$.
Then $\jac_p(f) = \jac_h(q) \circ \jac_e(q)^{-1}$.
Now $\jac_e(q)^{-1}$ is an isomorphism, so $\jac_p(f)$ is surjective if and only if $\jac_h(q)$ is surjective.
Hence (1) and (3) are equivalent and (2) implies (1).
\end{proof}

A \textbf{Nash isomorphism} is a Nash map with a Nash inverse.
Note that a Nash isomorphism is a homeomorphism.
We now give the inverse function theorem for Nash maps.

\begin{proposition}\label{prop:nash-inverse}
Let $f$ and $p$ be as in Lemma~\ref{lem:nash sub}.
Then the following are equivalent:
\begin{enumerate}[leftmargin=*]
\item $f$ is a Nash isomorphism at $p$.
\item $\jac_f(p)$ is invertible.
\item There is Nash data $(X, e, h, P)$ for $f$ such that $h$ is \'etale at $q = (e|_P)^{-1}(p)$.
\item If $(X, e, h, P)$ is Nash data for $f$ then $h$ is \'etale at $q = (e|_P)^{-1}(p)$.
\end{enumerate}
\end{proposition}

\begin{proof}
The equivalence of (2), (3), and (4) follows by a slight variation of the proof of Lemma~\ref{lem:nash sub}.
We show that (4) implies (1).
Suppose $(X, e, h,P)$ is Nash data for $f$ and $h$ is \'etale at $q = (e|_P)^{-1}(p)$.
Then $h$ gives a homeomorphism $P^* \to U$ for some \'etale-open neighborhoods $q \in P^* \subseteq P$ and $f(p) \in U \subseteq O_2$.
Hence $(X, h, e, P^*)$ is Nash data for a Nash map, and this Nash map is a local inverse for $f$ at $p$.
Therefore (1) holds.
We finally show that (1) implies (2).
Suppose $f$ is a Nash isomorphism at $p$, let $g$ be the local inverse of $f$ at $p$.
Then $g$ is a Nash map, and $\jac_f(p) \circ \jac_g(f(p))$ is the identity.
Hence $\jac_f(p)$ is invertible.
\end{proof}

% \begin{proposition}
% \label{prop:nash-inverse}
% Suppose that $V,V^*$ are smooth $K$-varieties, $U$ is a nonempty open subset of $V(K)$, $f \colon U \to W^*(K)$ is a Nash map, $p \in U$, and $\jac_f(p)$ is invertible.
% Then there are open neighborhoods $p \in O \subseteq U$  and $f(p) \in O^* \subseteq W(K)$ such that $f$ gives a Nash isomorphism $O \to O^*$.
% \end{proposition}

% \begin{proof}
% Fix Nash data $(X,e,h,P)$ for $f$.
% Let $q$ be the unique element of $P$ such that we have $e(q) = p$.
% Note that $T_p V$, $T_{q} X$, and $T_{f(p)} V^*$ all have the same dimension.
% We factor $\jac_f(p)$ as $\jac_h(q) \circ \jac_e(q)^{-1}$.
% Then $\jac_f(p)$ and $\jac_e(q)^{-1}$ are isomorphisms, hence $\jac_h(q)$ is an isomorphism, hence $h$ is \'etale at $q$.
% By Fact~\ref{fact:etale-homeo} $f$ is a composition of a homeomorphism and a local homeomorphism and is hence a local homeomorphism.
% The local inverse of $f$ is the composition of $e$ with the inverse of some , and is hence also Nash.
% \end{proof}

% We declare $\jac_{f,y}$ to be the $n\times n$ Jacobian matrix with $ij$-entry $\der f_i/\der y_j$.
% Note that $\jac_{f,y}$ is the composition of $\jac_f \colon O \to M_{ (m + n) \times n}(K)$ with a projection $M_{(m + n) \times n} \to M_{n \times n}(K)$, and is hence Nash by Proposition~\ref{prop:nash-diff}.

Let $f(x_1,\ldots,x_n,y_1,\ldots,y_n)$ be a Nash map $O \to K^{n}$ for $O$ an open subset of $K^{m + n}$.
We let $\jac_{f,y}(a,b)$ be the $n\times n$ Jacobian matrix with $ij$-entry $\der f_i/\der y_j(a,b)$.

\begin{proposition}
\label{prop:nash-implicit}
Suppose that $O$ is a nonempty open subset of $K^{m + n}$, $f \colon O \to K^n$ is Nash, $(a,b) \in K^{m + n}$ satisfies $f(a,b) = 0$, and $\jac_{f,y}(a,b)$ is invertible.
Then there are open neighborhoods $a \in P \subseteq K^m$, $b \in S \subseteq K^n$ and a Nash map $s \colon P \to S$ such that $s(a) = b$ and $f(a^*,s(a^*)) = 0$ for all $a^* \in P$.
\end{proposition}

We follow a standard proof of the implicit function theorem over $\R$ and omit some details.

\begin{proof}
Let $h \colon O \to K^{m + n}$ be given by $h(x,y) = (x,f(x,y))$.
An argument identical with the usual case over $\R$ shows that $\jac_h(a, b)$ is invertible.
By the Nash inverse function theorem there are \'etale-open neighborhoods $(a,b) \in O \subseteq K^{m + n}$, $(a,0) \in O^* \subseteq K^{m + n}$ such that $h$ gives a Nash isomorphism $O \to O^*$.
Thus for every $(a^*,0) \in O^*$ and $(c,d) \in O$ we have 
$$h(c,d) = (a^*,0) \quad\Longleftrightarrow\quad (c,f(c,d)) = (a^*,0) \quad\Longleftrightarrow\quad c = a^* \text{  and  } f(a^*,d) = 0.$$
Let $P$ be the set of $a^* \in K^m$ such that $(a^*,0) \in O^*$.
Then $P$ is an \'etale-open neighborhood of $a$ by Fact~\ref{fact:slice}(1).
For any $a^* \in P$ there is a unique $(c,d) \in O$ such that $h(c,d) = (a^*,0)$, hence $c = a^*$ and $f(a^*,d) = 0$.
Let $S$ be the set of $d \in K^n$ such that $(a^*,d) \in O$ for some $a^* \in P$, this is an \'etale-open neighborhood of $b$ by Fact~\ref{fact:slice}(2).
Let $s\colon P \to S$ be given by  $s(a^*) = \uppi((h|_O)^{-1} (a^*,0))$, where $\uppi$ is the projection $K^{m + n} \to K^n$.
% By construction $s$ is unique.
Finally, $s$ is a composition of three Nash maps and is hence a Nash map.
\end{proof}

% We finally show that, roughly speaking, all Nash maps $K^n \to K^m$ arise from the implicit function theorem.

% \begin{proposition}
% Let $O, O^*$ be a nonempty \'etale-open subset of $K^m, K^n$, respectively, and let $f\colon O \to O^*$ be a Nash map.
% Then there is a polynomial $h \in K[x_1,\ldots,x_m,y_1,\ldots,y_n]$
% % and an \'etale-open subset $P$ of $\{ (a,b) \in K{m + n} : f(a,b) = 0 \}$
% such that 
% \begin{enumerate}[leftmargin=*]
% \item $\jac_{h,y}(a,b)$ is invertible for every $(a,b) \in P$.
% \item $h(a,f(a)) = 0$ for all $a \in O$.
% % \item The projection $\uppi_1 \colon P \to O$ is a homeomorphism.
% % \item $f = \uppi_2 \circ \uppi^{-1}$ for $\uppi_2$ the projection $K^{m + n} \to K^n$.
% \end{enumerate}
% \end{proposition}

It is also possible to prove the inverse and implicit functions theorems for Nash maps by applying Fact~\ref{fact:alg power series} and results on algebraic power series.

\subsection{Nash submersions}
We now wish to prove analogues of the standard local description of the diffeomorphism type of a germ of a submersion between smooth manifolds, i.e. we show that a Nash submersion is locally Nash isomorphic to a projection.

\begin{lemma}
\label{lem:submersion0}
Suppose $V,W$ are smooth $K$-varieties, $O$ is a non-empty open subset of $V(K)$, and $f \colon O \to W(K)$ is a Nash map.
Then the following are equivalent:
\begin{enumerate}
\item $f$ is a submersion.
\item if $(X,e,h,P)$ is Nash data for $f$ then $h$ is smooth on an open subvariety $U$ of $X$ containing $P$.
\item there is Nash data $(X,e,h,P)$ for $f$ such that $h$ is a smooth morphism.
\end{enumerate}
\end{lemma}

Lemma~\ref{lem:submersion0} follows from Lemma~\ref{lem:nash sub} and the fact that the smooth locus of a morphism $V \to W$ is a Zariski open subvariety.
Fact~\ref{fact:smooth-open} is \cite[Prop.~3.2]{secondpaper}.

\begin{fact}\label{fact:smooth-open}
If $V \to W$ is smooth then $V(K) \to W(K)$ is an open map in the \'etale-open topology.
\end{fact}

Fact~\ref{fact:smooth-factor} is \cite[\textsection 2.2 Prop.~11]{Neron}.

\begin{fact}
\label{fact:smooth-factor}
Suppose that $f \colon V \to W$ is a smooth morphism, $p \in V$, and the relative dimension of $f$ at $p$ is $d \ge 1$.
Then there is an open subvariety $U$ of $V$ containing $p$ such that the restriction of $f$ to $U$ factors as $U \to W \times \A^d \to W$ where $U \to W \times \A^d$ is \'etale $W \times \A^d \to W$ is the projection.
\end{fact}

Lemma~\ref{lem:smooth-case} follows from Fact~\ref{fact:smooth-factor} and Fact~\ref{fact:etale-homeo}.

\begin{lemma}
\label{lem:smooth-case}
Suppose that $f \colon V \to W$ is a smooth morphism of irreducible $K$-varieties, $p \in V(K)$, and $f$ has relative dimension $d \ge 1$ at $p$.
Let $\uppi$ be the projection $W \times K^d \to W$.
Then there are \'etale-open neighborhoods $p \in O \subseteq V(K)$  and $f(p) \in P \subseteq W(K)$, an \'etale-open $O^* \subseteq W(K) \times K^d$ such that $f(p) \in \uppi(O^*) = P$, and a Nash isomorphism $O \to O^*$ such that the following diagram commutes.
\begin{center}
\begin{tikzcd}
                                & O^* \arrow[rd, "\uppi"] &   \\
O \arrow[rr, "f|_O"] \arrow[ru] &                       & P
\end{tikzcd}
\end{center}
\end{lemma}

We now give the local description of Nash submersions.

\begin{lemma}
\label{lem:smooth-case-1}
Let $V, W$ be smooth and irreducible, $D$ be a nonempty open subset of $V(K)$, and $f \colon D \to W(K)$ be a Nash submersion.
Set $m = \dim V$ and $n = \dim W$ and let $\uppi$ be the projection $K^m \to K^n$ onto the first $n$ coordinates.
Then there are \'etale-open neighborhoods $p \in O \subseteq V(K)$ and $f(p) \in P \subseteq W(K)$,  \'etale-open subsets $O^* \subseteq K^m$ and $P^*\subseteq K^n$, and Nash isomorphisms $O \to O^*$ and $P \to P^*$, such that $P^* = \uppi(O^*)$ and the following commutes.
\[ \xymatrix{
O \ar[r]^{f} \ar[d] & \ar[d] P  \\ O^*
\ar[r]^{\uppi} & P^*}\]
\end{lemma}

\begin{proof}
Applying Lemma~\ref{lem:smooth-case} fix \'etale-open $p \in O \subseteq V(K)$ and $f(p) \in P \subseteq W(K)$, an \'etale-open $O^{**} \subseteq W(K) \times K^d$ such that $f(p) \in \uppi(O^{**}) = P$, and a Nash isomorphism $O \to O^{**}$ such that $O \to P$ factors as $O \to O^{**} \to P$, where $O^{**} \to P$ is the projection.
As $W$ is smooth there is an open subvariety $U$ of $W$ containing $p$ and an \'etale morphism $e \colon W \to \A^n$.

\begin{center}
\begin{tikzcd}
O \arrow[d] \arrow[r] & O^{**} \arrow[ld] \arrow[r] & O^* \arrow[d] \\
P \arrow[rr]          &                             & P^*           
\end{tikzcd}
\end{center}

After applying Fact~\ref{fact:etale-homeo} and possibly shrinking $P$ we may suppose that $e$ restricts to an open embedding $P \to K^n$, let $P^* = e(P)$, and let $P \to P^*$ be the restriction of $e$.
So $P \to P^*$ is a Nash isomorphism.
Let $h \colon U \times \A^{m - n} \to \A^n \times \A^{m - n}$ be the morphism $h(x,y) = (e(x), y)$.
Then $h$ is \'etale by an easy Jacobian computation.
After shrinking everything we may suppose that $h$ restricts to an open embedding $O^{**} \to K^m$.
Let $O^* = h(O^{**})$ and let $O \to O^*$ be the composition of $O \to O^{**}$ with $h$.
Then $O \to O^*$ is a composition of two Nash isomorphisms and is hence a Nash isomorphism.
\end{proof}

Fact~\ref{fact:smooth-fiber} follows from \cite[II Thm.~8.17]{hartshorne}.

\begin{fact}\label{fact:smooth-fiber}
Suppose that $V$ is a smooth $m$-dimensional irreducible $K$-variety, $W$ is a smooth $n$-dimensional subvariety of $V$, and $p \in W$.
Then there is an open subvariety $U$ of $V$ containing $p$ and a smooth morphism $f \colon U \to \A^{m - n}$ such that $W \cap U = f^{-1}(0)$.
\end{fact}

We now show that the $K$-points a smooth subvariety is locally a ``leaf" of a ``foliation" that is locally Nash isomorphic to a ``foliation" of $K^n$ by hyperplanes.

\begin{proposition}\label{prop:D}
Suppose that $V$ is an $m$-dimensional smooth irreducible $K$-variety, $W$ is a smooth $n$-dimensional subvariety of $V$, and $p \in W(K)$.
Consider $K^n$ to be a subspace of $K^m$ via the inclusion $(x_1,\ldots,x_n) \mapsto (x_1,\ldots,x_n,0,\ldots,0)$ and let $\upiota$ be the closed immersion $W \to V$.
Then there is an \'etale-open neighborhood $O \subseteq V(K)$ of $p$, an \'etale-open neighborhood $P \subseteq K^m$ of the origin, and Nash isomorphisms $O \cap W(K) \to P \cap K^n$ and $O \to P$ both taking $p$ to the origin, such that the following diagram commutes 
\[ \xymatrix{
O \cap W(K) \ar[r]^{\quad\quad \upiota} \ar[d] & \ar[d]  O  \\ P \cap K^n
\ar[r] & P}\]
where $P \cap K^n \to P$ is the inclusion.
\end{proposition}

\begin{proof}
Applying Fact~\ref{fact:smooth-fiber} we let $U$ be an open subvariety of $V$ containing $p$ and fix a smooth morphism $f \colon U \to \A^{m - n}$ such that $W \cap U = f^{-1}(0)$.
By Fact~\ref{fact:smooth-open} $f$ gives an open map $U(K) \to K^{m - n}$.
Applying Lemma~\ref{lem:smooth-case} we obtain an \'etale-open neighborhood $p \in O \subseteq V(K)$, an \'etale-open neighborhood $P\subseteq K^m$, and a Nash isomorphism $O \to P$, such that the map $O \to f(O)$ given by $f$ factors as $O \to P \to f(O)$, where $P \to f(O)$ is the restriction of the projection $K^m \to K^{m - n}$.
Let $O \cap W(K) \to P \cap K^n$ be the restriction of $O \to P$ to $O \cap W(K)$.
\end{proof}

\section{Characterizations of largeness}\label{section:char}
We now drop the assumption that $K$ is large and not separably closed and assume that $K$ is an arbitrary field.
A {\bf system of topologies} over $K$ is a choice of topology on $V(K)$ for every $V$ such that:
\begin{enumerate}
\item $V(K) \to W(K)$ is continuous for any $V \to W$.
\item $V(K) \to W(K)$ is a topological closed embedding when $V \to W$ is a scheme-theoretic closed immersion.
\item $V(K) \to W(K)$ is a topological open embedding when $V \to W$ is a scheme-theoretic open immersion.
\end{enumerate}
Any field topology on $K$ induces a system of topologies over $K$.
(We assume that all field topologies are hausdorff.)
We will not distinguish between a field topology and the system that it induces.
The classical examples of systems of topologies are the Zariski topology and those induced by field topologies.
The \'etale-open topology over $K$ is a system of topologies.
% If $\cT$ is a system of topologies over $K$ then the $\cT$-topology on $V(K) \times W(K) = (V \times W)(K)$ always refines the product of the $\cT$-topologies on $V(K)$ and $W(K)$, and $\cT$ is a field topology if and only if the $\cT$-topology on $V(K) \times W(K)$ always agrees with the product topology.
% Hence we can view systems of topologies as a natural generalization of the notion of a field topology.
A system of topologies over $K$ is discrete if it satisfies one of the following equivalent conditions.
\begin{enumerate}[leftmargin=*]
\item The topology on $K$ is discrete.
\item There is a nonempty finite open subset of $K$.
\item The topology on $V(K)$ is discrete for every $V$.
\end{enumerate}
See \cite{firstpaper} for the equivalence of these conditions and other basic facts about systems of topologies.
We say that a system of topologies over $K$ {\bf satisfies the polynomial implicit function theorem} if whenever $f \in K[x_1,\ldots,x_m,y_1,\ldots,y_n]$ and $(a,b) \in K^m \times K^n$ satisfy $f(a,b) = 0$ and $\jac_{f,y}(a,b)$ is invertible then there are open neighborhoods $a \in P \subseteq K^m$ and $b \in S \subseteq K^n$  and continuous $s \colon P \to S$ such that $s(a) = b$ and $f(a^*,s(a^*)) = 0$ for all $a^* \in P$.

\begin{proposition}\label{prop:char}
Suppose that $K$ is not the algebraic closure of a finite field.
Then the following are equivalent.
\begin{enumerate}[leftmargin=*]
\item $K$ is large.
\item There is a non-discrete system of topologies over $K$ such that $V(K) \to W(K)$ is a local homeomorphism when $V \to W$ is \'etale.
\item There is a non-discrete system of topologies over $K$ that satisfies the polynomial implicit function theorem.
\end{enumerate}
\end{proposition}

\begin{proof}
Propositions~\ref{prop:nash-inverse} and \ref{prop:nash-implicit} show that (1) implies both (2) and (3) when $K$ is not separably closed.
Suppose that $K$ is separably closed.
As $K$ is not the algebraic closure of a finite field there is a non-trivial valuation $v$ on $K$ and $v$ is henselian as any valuation on a separably closed field is henselian.
Then the system of topologies over $K$ induced by $v$ satisfies (2) and (3) by \cite[Fact~6.6]{firstpaper} and \cite[Cor.~3.3.17]{trans}, respectively.

\medskip
Let $\cT$ be a system of topologies over $K$.
Suppose that $\cT$ satisfies (2).
Then $f(V(K))$ is $\cT$-open when $f\colon V \to W$ is \'etale, hence $\cT$ refines the \'etale-open topology, hence the \'etale-open topology is not discrete, hence $K$ is large.

\medskip
Finally, suppose that $\cT$ satisfies (3).
Suppose that $f \in K[x,y]$ and $(a,b) \in K^2$ are such that $f(a,b) = 0 \ne \der f/\der y (a,b)$.
There is a $\cT$-open neighborhood $O$ of $a$ so that $f(a^*,y)$ has a root in $K$ for every $a^* \in O$.
Then $O$ is infinite as $\cT$ is not discrete, hence $K$ is large.
\end{proof}

\begin{lemma}\label{lem:syslem}
Suppose that $K$ is not separably closed and $\cT$ is a system of topologies over $K$.
Then the following are equivalent:
\begin{enumerate}[leftmargin=*]
\item $V(K) \to W(K)$ is open when $V \to W$ is \'etale.
\item $V(K) \to W(K)$ is a local homeomorphism when $V \to W$ is \'etale.
\end{enumerate}
\end{lemma}

If $K$ is separably closed and $\cT$ is the Zariski topology then (1) holds and (2) fails.

\begin{proof}
It is clear that (2) implies (1).
Suppose that (1) holds.
Then $\cT$ refines the \'etale-open topology.
Fix \'etale $V \to W$.
Then $V(K) \to W(K)$ is continuous and open, hence the restriction of $V(K) \to W(K)$ to some open subset $O$ of $V(K)$ is a homeomorphism onto its image if and only if it is injective.
By Fact~\ref{fact:etale-homeo} $V(K) \to W(K)$ is locally injective in the \'etale-open topology and in hence locally injective in $\cT$.
\end{proof}

Let $\uptau$ be a field topology on $K$.
Recall that $\uptau$ is {\bf gt-henselian} if for every $n$ and neighborhood $P$ of $-1$ there is a neighborhood $O$ of $0$ such that $x^{n + 1} + x^n + a_{n - 1}x^{n - 1} + \cdots + a_1 x + a_0$ has a root in $P$ for any $a_0,\ldots,a_{n - 1} \in O$.
Equivalently $\uptau$ is gt-henselian if $V(K) \to W(K)$ is $\uptau$-open when $V \to W$ is \'etale~\cite[Prop.~8.6]{field-top-2}.
Note that if follows that any gt-henselian field topology  refines the \'etale-open topology.

\begin{lemma}\label{lem:char}
Suppose that $K$ is not separably closed and $\uptau$ is a field topology on $K$.
Then the following are equivalent.
\begin{enumerate}[leftmargin=*]
\item $\uptau$ is gt-henselian.
\item $V(K) \to W(K)$ is a $\uptau$-local homeomorphism when $V \to W$ is \'etale.
\item $\uptau$ satisfies the polynomial implicit function theorem.
\end{enumerate}
\end{lemma}

This should also hold in the separably closed case, but this would require a proof that does not use the \'etale-open topology.

\begin{proof}
Lemma~\ref{lem:syslem} and the remarks above show that (1) and (2) are equivalent.
It is easy to see that (3) implies (1) by considering the polynomial $$f(x,y_1,\ldots,y_{n-1}) = x^{n+1} + x^n + y_{n - 1}x^{n - 1} + \cdots + y_1 x  + y_0$$
and noting that $\der f/\der x(-1, 0,\ldots,0) = (-1)^{n - 1} \ne 0$.

\medskip
Suppose that $\uptau$ is gt-henselian.
Then $\uptau$ refines the \'etale-open topology.
Suppose furthermore that $f \in K[x_1,\ldots,x_m,y_1,\ldots,y_n]$ and $(a,b) \in K^m \times K^n$ are such that $f(a,b) = 0$ and $\jac_{f,y}(a,b)$ is invertible.
By Proposition~\ref{prop:nash-implicit} there are \'etale-open neighborhoods $P,S$ and a Nash map $s \colon P \to S$ such that $s(a) = b$ and $f(a^*,s(a^*)) = 0$ for all $a^* \in P$.
Now $P$ and $S$ are $\uptau$-open as $\uptau$ refines the \'etale-open topology.
Let $(X,e,h,P^*)$ be Nash data for $s$, so $s = h \circ (e|_{P^*})^{-1}$.
Now $e$ is \'etale and invertible on $P^*$, so $e|_{P^*}$ is an injective continuous open map and hence a homeomorphism with respect to $\uptau$.
Therefore $s$ is a composition of two continuous maps and is hence continuous with respect to $\uptau$.
Hence (1) implies (3).
\end{proof}

\begin{corollary}\label{cor:char}
The following are equivalent:
\begin{enumerate}[leftmargin=*]
\item $K$ is large.
\item Some elementary extension of $K$ admits a non-discrete gt-henselian field topology.
\item There is an elementary extension $\K$ of $K$ and a non-discrete field topology $\uptau$ on $\K$ such that $V(\K) \to W(\K)$ is a local homeomorphism with respect to $\uptau$ for any \'etale morphism $V \to W$ of $\K$-varieties.
\item There is an elementary extension $\K$ of $K$ and a non-discrete field topology $\uptau$ on $\K$ that satisfies the polynomial implicit function theorem.
\end{enumerate}
In fact, we may take $\K$ to be any $\aleph_1$-saturated elementary extension of $K$.
\end{corollary}

It is an open question whether every large field admits a gt-henselian field topology.

\begin{proof}
Let $\K$ be an $\aleph_1$-saturated elementary extension of $K$.
Lemma~\ref{lem:char} shows that (2), (3), and (4) are equivalent.
By Proposition~\ref{prop:char} (3) implies that $\K$ is large and hence implies that $K$ is large as large fields form an elementary class.
Suppose that $K$ is large and let $\K$ be an $\aleph_1$-saturated elementary extension of $K$.
By \cite[Thm.~A]{large->henselian} there is a henselian local subring $R$ of $\K$ such that $\K$ is the fraction field of $R$.
Let $\uptau$ be the $R$-adic topology on $\K$, i.e. the topology with basis sets of the form $aR + b$ for $a \in \K^\times, b \in \K$.
Then $\uptau$ is non-discrete as $R$ is infinite and $\uptau$ is a gt-henselian field topology by \cite[Prop.~8.5]{field-top-2}.
\end{proof}

\section{Definable functions in \'ez fields}
\label{section:ez}

A subset of $V(K)$ is \'ez if it is a finite union of definable \'etale-open subsets of Zariski closed sets.
If $K$ is not large then every subset of $V(K)$ is trivially \'ez as the \'etale-open topology is discrete.
We say that $K$ is {\bf \'ez} if it is large and every definable subset of every $K^m$ is \'ez.
We are only aware of one obstruction to being \'ez among large fields: existence of a proper infinite definable subfield.
In particular that it follows that \'ez fields are perfect.
Any perfect large field that admits quantifier elimination in some expansion of the language of rings by constants is \'ez.
In particular pseudofinite fields and infinite algebraic extensions of finite fields are \'ez.
See \cite{secondpaper} for background and more examples of \'ez fields.

\medskip
We show that definable functions in characteristic zero \'ez fields are generically locally Nash.

\begin{theorem}
\label{thm:ez-function}
Suppose that $K$ is \'ez, $\Chara(K) = 0$, $V$ and $W$ are  $K$-varieties with $V$ smooth irreducible, $O$ is a nonempty definable \'etale-open subset of $V(K)$, and $f \colon O \to W(K)$ is definable.
Then there is a dense open subvariety $U$ of $V$ and definable \'etale-open subsets $O_1,\ldots,O_k$ of $O$ such that $O \cap U(K) = \bigcup_{i = 1}^{k} O_i$ and the restriction of $f$ to each $O_i$ is Nash.
\end{theorem}

By Fact~\ref{fact:z-dense} $O \cup U(K)$ is Zariski dense in $O$.
Theorem~\ref{thm:ez-function} fails in positive characteristic.

\begin{lemma}
\label{lem:1/p}
Suppose that $K$ is large and perfect of characteristic $p \ne 0$.
Then the map $K \to K$ given by $a \mapsto a^{1/p}$ is not locally Nash at any $b \in K$.
\end{lemma}

\begin{proof}
If $a \mapsto a^{1/p}$ is locally Nash at $b \in K$ then it has a derivative at $b$ and this is the inverse of the derivative of the Frobenius at $b^{1/p}$, which is zero, contradiction.
\end{proof}

Fact~\ref{fact:def-func} is \cite[Prop.~10.3]{secondpaper}.

\begin{fact}
\label{fact:def-func}
Suppose that $K$ is \'ez, $V$ is a smooth irreducible subvariety of $\A^m$, $O$ is a nonempty definable \'etale-open subset of $V(K)$, and $f \colon O \to K$ is definable.
Then there is a dense open subvariety $U$ of $V$, definable \'etale-open subsets $O_1,\ldots,O_k$ of $O$, and irreducible polynomials $h_1,\ldots,h_k \in K[x_1,\ldots,x_m,t]$ such that $O \cap U(K) = \bigcup_{i = 1}^{k} O_i$ and for every $i =1,\ldots,k$:
\begin{enumerate}
%\item $O \cap U(K) = \bigcup_{i = 1}^{k} O_i$,
\item $h_i(a,f(a)) = 0$ and $h_i(a,t)$ is not constant zero for all $a \in O_i$,
\item the closed subvariety $W_i$ of $U \times \A^1$ given by $h_i(x_1,\ldots,x_m,t) = 0$ is smooth,
\item the graph of the restriction of $f$ to $O_i$ is an \'etale-open subset of $W_i(K)$,
\item $f$ is continuous on $O_i$.
\end{enumerate}
\end{fact}

Fact~\ref{fact:sard} fails in positive characteristic (consider the Frobenius $\A^1 \to \A^1$).

\begin{fact}\label{fact:sard}
Suppose that $K$ is characteristic zero and $V \to W$ is quasi-finite.
Then there is a dense open subvariety $U$ of $W$ such that the pullback of $V \to W$ by $U \to W$ is \'etale.
\end{fact}

\begin{proof}
By \cite[Cor.~5.4.2]{Mumford-Oda} there is a dense open subvariety $U^*$ of $V$ such that $U^* \to W$ is \'etale.
As $V \to W$ is quasi-finite the image of $V \setminus U^*$ is not Zariski dense in $W$.
Let $U$ be the complement of the Zariski closure of the image of $V \setminus U^*$.
\end{proof}

We now prove Theorem~\ref{thm:ez-function}.

\begin{proof}
After possibly replacing $V$ with a dense affine open subvariety we suppose that $V$ is a subvariety of $\A^m$.
After possibly replacing $W$ with the Zariski closure of $f(O)$ we may suppose that $f(O)$ is Zariski dense in $W$.
Fix a dense affine open subvariety $W^*$ of $W$, so $f^{-1}(W^*)$ is Zariski dense in $V$.
Hence $f^{-1}(W^*(K)) \cap O$ is Zariski dense in $O$ by Fact~\ref{fact:z-dense}.
So we may suppose that $W$ is affine.
We may therefore also suppose that $W$ is a subvariety of $\A^n$ and thereby reduce to the case when $W = \A^n$.
An application of Fact~\ref{fact:nash basic}(2) further reduces to the case when $W = \A^1$.

\medskip
Let $U$, $O_1,\ldots,O_k$, and $h_1,\ldots,h_k$ be as in Fact~\ref{fact:def-func}.
It suffices to show that the restriction of $f$ to each $O_i$ is generically Nash.
Hence we may suppose that $k = 1$, set $O = O_1$, and set $h = h_1$.
Let $Y$ be the subvariety of $U \times \A^1$ given by $h = 0$ and let $\uppi\colon Y \to U$ be the projection.
Then $\uppi$ is quasi-finite by (1).
Applying Fact~\ref{fact:sard} let $U^*$ be a dense open subvariety of $U$ such that the pullback $\uppi^*$ of $\uppi$ by $U^* \to U$ is \'etale.
Let $O^* = U^*(K)$.
We show that $f$ is Nash on $O^*$.
Let $G$ be the graph of $f|_{O^*}$.
By (3) above $G$ is an \'etale-open subset of $W(K)$.
The restriction of $\uppi^*$ to $G$ is a injection and is hence a homeomorphism by Fact~\ref{fact:etale-homeo}.
Let $\rho \colon W \to \A^1$ be the restriction of the projection $V \times \A^1 \to \A^1$ to $W$.
Then $(W, \uppi^*, \rho, G)$ is Nash data for $f|_{O^*}$, hence $f|_{O^*}$ is a Nash map.
\end{proof}

\section{Krasner's lemma for large fields}\label{section:kras}
Given a morphism $V \to W$ we let $V_p$ be the scheme-theoretic fiber of $V$ over $p\in W$, this is a scheme of finite type over $K$.
If $V \to W$ is \'etale and $p$ is a $K$-point then $V_p$ is a finite disjoint union of spectra of finite separable extensions of $K$.

\begin{proposition}\label{prop:krasner}
Suppose that $c\colon V \to W$ is finite and \'etale.
Then the isomorphism type of $V_p$ is locally constant in the \'etale-open topology on $W(K)$.
\end{proposition}

Poonen \cite[Prop.~3.5.74]{poonen-qpoints} gives a proof of this for the analytic topology over a local field $K$.
Our proof follows his.
(The \'etale-open topology over a local field other than $\C$ agrees with the analytic topology.)

\medskip
Fact~\ref{fact:insep} is \cite[Thm.~5.8, Prop.~5.10]{firstpaper}.

\begin{fact}\label{fact:insep}
Suppose that $L/K$ is an algebraic field extension, fix a $K$-variety $V$, and let $V(K) \to V(L)$ be the inclusion.
Then $V(K) \to V(L)$ is a continuous map from the \'etale-open topology on $V(K)$ to the \'etale-open topology on $V(L)$ and if $L/K$ is purely inseparable then $V(K) \to V(L)$ is a homeomorphic embedding.
\end{fact}

We now prove Proposition~\ref{prop:krasner}.

\begin{proof}
We may suppose that $W$ is connected, hence $e$ has constant degree $d$.
If $K$ is separably closed then each $V_p$ is a disjoint union of $d$ copies of $\Spec K$, so we may suppose that $K$ is not separably closed.
The case when $K = \R$ follows is a special case of \cite[Prop.~3.5.74]{poonen-qpoints} and the case when $K$ is real closed then follows by elementary transfer.
We suppose that $K$ is not real closed.
Fix $p \in W(K)$.
Let $W'$ be an affine open subvariety of $W$ containing $p$.
After possibly replacing $W$ with  $W'$ and $V \to W$ with the pullback of $V \to W$ by $W' \to W$ we may suppose that $W$ is affine.
As $e$ is finite, $e$ is affine, hence $V$ is affine.
We show there is an \'etale-open neighborhood $O \subseteq W(K)$ of $p$ such that $V_q$ is isomorphic to $V_p$ for all $q \in O$.
Let $L/K$ be a finite Galois extension such that $V_p$ splits into points $p_1,\ldots,p_d \in V(L)$.
Then $L$ is not separably closed as $K$ is neither separably nor real closed.
Let $e_L$ be the induced map $V(L) \to W(L)$, so $e^{-1}_L(p) = V_p(L) = \{p_1,\ldots,p_d\}$.
Now $e_L$ is a local homeomorphism as $L$ is not separably closed.
Let $O \subseteq W(L)$ be an \'etale-open neighborhood of $p$ and $P_i \subseteq V(L)$ be an \'etale-open neighborhood of $p_i$ such that $e_L$ gives a homeomorphism $P_i \to O$ for each $i = 1,\ldots,d$.
Now $\Gal(L/K)$ acts on $V(L)$ and by \cite[Cor.~2.7]{secondpaper} this action is by homeomorphisms with respect to the \'etale-open topology.
After possibly shrinking these neighborhoods we may suppose that $\sigma(P_i) = P_j$ when $\sigma \in \Gal(L/K)$ is such that $\sigma(p_i) = p_j$.
As $V$ is affine the \'etale-open topology on $V(L)$ is hausdorff by Fact~\ref{fact:hd} so we may also suppose that the $P_i$ are pairwise disjoint.
Let $P$ be the union of the $P_i$ and let $h$ be the natural map $P \to O \times e^{-1}_L(p)$.
The action of $\Gal(L/K)$ on $e^{-1}_L(p)$ induces a natural action on $O \times e^{-1}_L(p)$, and $h$ is then $\Gal(L/K)$-equivariant.
By Fact~\ref{fact:insep} $O \cap V(K)$ is an \'etale-open subset of $V(K)$.
Fix $q \in O \cap V(K)$.
Then $e^{-1}_L(p)$ and $e^{-1}_L(q)$ are isomorphic as $\Gal(L/K)$-sets.
Now $V_q$ also splits over $L$ as $|e^{-1}_L(q)| = d$ and $V_q$ has degree $d$.
Hence the isomorphism $e^{-1}_L(p) \to e^{-1}_L(q)$ of $\Gal(L/K)$-sets gives an isomorphism $V_p \to V_q$ of $K$-schemes.
\end{proof}

Given $\alpha = (\alpha_0,\ldots,\alpha_d) \in K^{d + 1}$ let $p_\alpha \in K[x]$ be $p_\alpha(x) = x^{d + 1} + \alpha_d x^d + \cdots + \alpha_1 x + \alpha_0$.

\begin{proposition}\label{cor:krasner}
Fix $d \ge 1$ and $\alpha \in K^{d +1}$ and suppose that $p_\alpha$ is separable and irreducible.
Then there is an \'etale-open neighborhood $O$ of $\alpha$ such that if $\beta \in O$ then $p_\beta$ is separable irreducible and $K[x]/(p_\beta)$ is isomorphic to $K[x]/(p_\alpha)$.    
\end{proposition}

Proposition~\ref{cor:krasner} was first proven in \cite{with-anand}.
We give a more conceptually satisfying proof.

\begin{fact}\label{fact:fc}
If $K$ is perfect and $V \to W$ is a finite morphism then the image of $V(K) \to W(K)$ is a closed subset of $W(K)$ in the \'etale-open topology.
\end{fact}

Fact~\ref{fact:fc} is \cite[Thm.~C]{large->henselian}.

\begin{fact}\label{fact:poly mult}
The morphism $\A^m \times \A^n \to \A^{m + n}$ given by $(p_\alpha, p_\beta) \mapsto p_\alpha p_\beta$ is finite.
\end{fact}

Fact~\ref{fact:poly mult} is proven in \cite[Fact~6.1]{large->henselian}, and has probably been proven in a number of other places as well.
We now prove Proposition~\ref{prop:krasner}.

\begin{proof}
Let $S$ be the set of $\beta \in K^{d + 1}$ such that $p_\beta$ is separable and irreducible.
Let $L/K$ be the maximal purely inseparable extension of $K$, equivalently $L$ is the smallest perfect extension of $K$.
Facts~\ref{fact:poly mult}, \ref{fact:fc}, and \ref{fact:insep} together show that the set of $\beta \in K^{d + 1}$ such that $p_\beta$ factors in $L[x]$ is closed.
Recall that a polynomial in $K[x]$ does not factor in $L[x]$ if and only if it is separable and irreducible.
Hence $S$ is \'etale-open.
Now let $V$ be the closed subvariety of $\A^{d + 1} \times \A^1 = \Spec K[y_0,\ldots,y_d,x]$ given by $x^{d + 1} + y_d x^d + \cdots + y_1 x + y_0 = 0$ and let  $\uppi \colon V \to \A^{d + 1}$ be the projection.
Then $\uppi$ is finite as $K[y_0,\ldots,y_d,x]/(x^{d + 1} + y_d x^d + \cdots + y_1 x + y_0)$ is a finite extension of $K[y_0,\ldots,y_d]$.
Let $W$ be the open subvariety of $V$ given by
$$ \frac{d}{dx}[x^{d + 1} + y_d x^d + \cdots + y_1 x + y_0] \ne 0.$$
Then the restriction of $\uppi$ to $W$ is \'etale.
Now $V \setminus W$ is a closed subvariety of $V$, hence $\uppi(V \setminus W)$ is a closed subvariety of $\A^{d+1}$ as $\uppi$ is finite.
Let $U$ be the complementary subvariety of $\uppi(V \setminus W)$, this is an open subvariety of $\A^{d+1}$.
Note that $S \subseteq U(K)$.
Let $\uppi^* \colon V^* \to U$ be the pullback of $\uppi$ along $U \to \A^{d + 1}$, so $\uppi^*$ is finite and \'etale.
Hence by Proposition~\ref{prop:krasner} there is an \'etale-open neighborhood $\alpha \in O \subseteq S$ such that the isomorphism type of the scheme-theoretic fiber of $\uppi^*$ is constant on $O$.
Finally, note that if $\beta \in O$ then the scheme-theoretic fiber of $\uppi^*$ above $\beta$ is a disjoint union of $d$ copies of $\Spec K[x]/(p_\beta)$.
\end{proof}

\bibliographystyle{amsalpha}
\bibliography{the}

\providecommand{\bysame}{\leavevmode\hbox to3em{\hrulefill}\thinspace}
\providecommand{\MR}{\relax\ifhmode\unskip\space\fi MR }
\providecommand{\MRhref}[2]{%
  \href{http://www.ams.org/mathscinet-getitem?mr=#1}{#2}
}
\providecommand{\href}[2]{#2}
\begin{thebibliography}{AvdDvdH17}

\bibitem[Art69]{Artin_approx}
M.~Artin, \emph{Algebraic approximation of structures over complete local rings}, Publications mathématiques de l’IHÉS \textbf{36} (1969), no.~1, 23–58.

\bibitem[AvdDvdH17]{trans}
Matthias Aschenbrenner, Lou van~den Dries, and Joris van~der Hoeven, \emph{Asymptotic differential algebra and model theory of transseries}, Annals of Mathematics Studies, vol. 195, Princeton University Press, Princeton, NJ, 2017. \MR{3585498}

\bibitem[BCR98]{real-algebraic-geometry}
J.~Bochnak, M.~Coste, and M-F. Roy, \emph{Real algebraic geometry}, Springer, 1998.

\bibitem[BLR90]{Neron}
Siegfried Bosch, Werner Lütkebohmert, and Michel Raynaud, \emph{Neron models}, Springer-Verlag, 1990.

\bibitem[DWY22]{field-top-2}
Philip Dittmann, Erik Walsberg, and Jinhe Ye, \emph{When is the \'etale open topology a field topology?}, arXiv preprint arXiv:2208.02398, to appear in Israel. J. Math. (2022).

\bibitem[Har77]{hartshorne}
Robin Hartshorne, \emph{Algebraic geometry}, Graduate Texts in Mathematics, vol.~52, Springer-Verlag, 1977.

\bibitem[JTWY24]{firstpaper}
Will Johnson, Chieu-Minh Tran, Erik Walsberg, and Jinhe Ye, \emph{The \'etale-open topology and the stable fields conjecture}, J. Eur. Math. Soc. (JEMS) \textbf{26} (2024), no.~10, 4033--4070. \MR{4768414}

\bibitem[JTWY25]{large->henselian}
\bysame, \emph{Large implies henselian}, manuscript, 2025.

\bibitem[MO15]{Mumford-Oda}
David Mumford and Tadao Oda, \emph{Algebraic geometry. {II}}, Texts and Readings in Mathematics, vol.~73, Hindustan Book Agency, New Delhi, 2015. \MR{3443857}

\bibitem[Poo17]{poonen-qpoints}
Bjorn Poonen, \emph{Rational points on varieties}, Graduate Studies in Mathematics, vol. 186, American Mathematical Society, Providence, RI, 2017. \MR{3729254}

\bibitem[Pop96]{pop-embedding}
Florian Pop, \emph{Embedding problems over large fields}, Ann. of Math. (2) \textbf{144} (1996), no.~1, 1--34. \MR{1405941}

\bibitem[Pop14]{Pop-little}
\bysame, \emph{Little survey on large fields---old \& new}, Valuation theory in interaction, EMS Ser. Congr. Rep., Eur. Math. Soc., Z\"{u}rich, 2014, pp.~432--463. \MR{3329044}

\bibitem[PW23]{with-anand}
Anand Pillay and Erik Walsberg, \emph{Galois groups of large simple fields}, Model Theory \textbf{2} (2023), 357--380.

\bibitem[Wal22]{topological_proofs}
Erik Walsberg, \emph{Topological proofs of results on large fields}, Comptes Rendus. Mathématique \textbf{360} (2022), no.~G11, 1187–1192.

\bibitem[WY23]{secondpaper}
Erik Walsberg and Jinhe Ye, \emph{\'{E}z fields}, J. Algebra \textbf{614} (2023), 611--649. \MR{4499357}

\end{thebibliography}

\end{document}